\documentclass[12pt]{amsart}
\usepackage{amsmath, amsthm, amssymb, hyperref, color}
\usepackage{setspace}
\usepackage{graphicx}
\usepackage{caption}
\usepackage{subcaption}
\usepackage{mathtools}
\usepackage{enumerate}
\usepackage{stackengine}
\usepackage{ dsfont }
\usepackage[all]{xypic}
\usepackage{verbatim}
\usepackage{upgreek}
\usepackage{chemfig, chemnum}
\usepackage{tikz-cd}
\usepackage[lined,commentsnumbered,ruled]{algorithm2e}
\theoremstyle{theorem}
\usepackage{latexsym,amsmath,amssymb,epsfig,graphicx}
\usepackage{caption}
\usepackage{cellspace}
\usepackage[utf8]{inputenc}
\usepackage{wasysym}
\usepackage{wrapfig}
\usepackage{float}
\newtheorem{thm}{Theorem}
\newtheorem{defn}{Definition}

\newtheorem{prop}{Proposition}

\newtheorem{lem}{Lemma}

\newcommand{\FF}{\mathbb{F}}
\newcommand{\PP}{\mathbb{P}}

\newcommand{\CC}{\mathbb{C} }
\newcommand{\ZZ}{\mathbb{Z}}

\newcommand{\A}{\mathbb{A}}

\begin{document}
\title{Formal Abel relations for curves in characteristic $p$}
\author{John B. Little}
\address{Professor {\it emeritus} of Mathematics\\
College of the Holy Cross, Worcester, MA 01610}
\email{\tt jlittle@holycross.edu}

\date{\today}

\begin{abstract}
In this working paper, we report recent work studying the form of Abel relations in the (generalized) Jacobians of reduced plane curves over an algebraically closed
field of characteristic $p$ from a formal power series point of view.  The ultimate goal (to be addressed in a subsequent paper) is to complete the work done in the author's PhD thesis from 
1980 and to establish a general characteristic $p$ form of the \emph{converse of 
Abel's theorem} discussed by Griffiths and used in the Lie-Wirtinger theorem on double translation manifolds and the  geometry of webs.  
\end{abstract}

\maketitle

\section{Introduction}

The classical version of Abel's theorem (see \cite{LitAbel,Grif}) implies
the following statement.  Let $C$ be an algebraic curve of degree $d$ in $\PP^2$ over $\CC$, possibly singular or reducible  but without multiple components.   
Let $f(x,y) = 0$ be an affine equation of $C$ and let $\Delta$ be a line meeting $C$ in $d$ distinct points.  By choosing coordinates in $\PP^2$, we may take $\Delta$
as the $y$-axis.  If $u,v$ are sufficiently small in absolute value, then the line $\Delta_{u,v}:  x = uy + v$ will again meet $C$ in $d$ distinct points
$P_i(u,v)$. 

\begin{thm}
\label{AbClassical}
 If $\omega$ is any $1$-form 
\begin{equation}
\label{AbDiff} 
\omega = \frac{p(x,y) dx}{\partial f/ \partial y},\ \deg(p(x,y)) \le d - 3
\end{equation}
on $C$, then integrating along paths lying in the smooth locus of $C$ the abelian sum 
\begin{equation}
\label{AbSum}
\sum_{i=1}^d \int_{P_i(0,0)}^{P_i(u,v)} \omega
\end{equation}
is constant modulo periods.  
\end{thm}

Motivated by the Lie-Wirtinger theorem on double translation manifolds and related applications in web geometry, 
in \cite{Grif}, and still working over $\CC$, Grifffiths also studied a sort of \emph{converse of this version of Abel's theorem}.

\begin{thm} 
\label{AbConv}
 Let $C_i$ be $d \ge 3$ analytic curves in a neighborhood of the $y$-axis in $\PP^2$ over $\CC$ meeting the $y$-axis
transversely in $d$ distinct points and let $\Delta_{u,v}$, $P_i(u,v)$ be as above.  If there exist $t_i : C_i \to \CC$ with $t_i'(0) \ne 0$
such that 
\begin{equation}
\label{AbRel}
\sum_{i=1}^d t_i(P_i(u,v)) = 0,
\end{equation}
identically in $u,v$, then the $C_i$ lie on an algebraic curve $C$ of degree $d$ and $dt_i = \omega|_{C_i}$ for some Abelian differential $\omega$ as in 
\eqref{AbDiff}.  
\end{thm}

The connections with double translation manifolds and characterizations of Jacobians are discussed in \cite{M} and \cite{Lit1}. 
The applications to web geometry are discussed in the classic book \cite{BB} of Blaschke and Bol and the more 
recent text \cite{PP} of Pereira and Pirio.

The idea is that \emph{addition laws} such as \eqref{AbRel} come only from integrating Abelian differentials on 
algebraic curves, or in more sophisticated terms, from the Jacobian varieties of curves.  One perhaps unexpected
feature here is that the curve $C$ might have singularities away from the neighborhood of $\Delta_0$.  In that 
case, $\omega$ will be a \emph{dualizing differential} on $C$ in the sense of Rosenlicht and the relation
will take place in the \emph{generalized Jacobian} of the curve.

All of the above can be rephrased purely algebraically so it makes sense to ask whether some statement like
Theorem~\ref{AbConv} is true in characteristic $p$ as well.   The author has studied the case $d = 3$ in characteristic
$p$ in \cite{Lit} and intends to return to this question in a subsequent paper.

\section{Formal Abel relations}

To prepare for an attack on the converse of Abel for general $d$,  in these notes,
we consider how the characteristic $0$ version of Abel's theorem in Theorem~\ref{AbClassical} goes over to local statements
valid over any algebraically closed field, including fields of characteristic $p$.  The translation goes as follows.  

Assuming $k$ is an algebraically closed field of characteristic $p$, let $\PP$ denote the formal completion
of $\PP^2$ along the line $\Delta_0$ defined by $x = 0$.   Let $C$ be an algebraic curve of degree $d$ with an affine
equation of the form 
$$0 = f(x,y) = y^d + c_1(x) y^{d-1} + \cdots + c_{d-1}(x) y + c_d(x)$$
where $\deg c_i(x) \le i$.  Assume that the equation 
$$0 = f(0,y) = y^d + c_1(0) y^{d-1} + \cdots + c_{d-1}(0) y + c_d(0)$$
has $d$ distinct roots $a_{i,0}$ in $k$.  
Then the formal implicit function theorem implies that there are $d$ formal curves $C_i$ defined by 
$y = f_i(x) \in k[[x]]$ that give the local expansions of $C$ at the points $(0,a_{i,0})$.  
In other words, in the ring $k[[x]][y]$, we have a factorization
$$f(x,y) = (y - f_1(x))(y - f_2(x)) \cdots (y - f_d(x)),$$
where 
$$f_i(x) = \sum_{n=0}^\infty a_{i,n} x^n$$
for some $a_{i,n} \in k$.

For simplicity, we will mostly concentrate on the case where the projective closure curve of the $C$ in $\PP^2$
is a smooth curve of degree $d$, hence of genus $g = (d - 1)(d - 2)/2$.

Let ${\rm Spf}\, k[[u,v]]$
represent the formal completion of the dual projective space $(\PP^2)^\vee$ at the point corresponding to $\Delta_0$.
We again parametrize lines ``near the $y$-axis'' as lines $\Delta_{u,v}$ defined by $x = uy + v$.  By the formal 
implicit function theorem again, we can formally solve the equations $y = f_i(uy + v)$ for $y = y_i(u,v)$, then 
take $x_i(u,v) = u y_i(u,v) + v$ to obtain the power series for 
$P_i(u,v) = \Delta_{u,v} \cap C_i$ as 
$$P_i(u,v) = (x_i(u,v), y_i(u,v)).$$

Finally, the sum in \eqref{AbRel} must be replaced by a more general rule of combination.  One (maximal) choice is the 
\emph{commutative formal group law} over $k$ corresponding to the group operation in the (generalized) Jacobian $J(C)$.  In more 
detail, taking a suitable set of local coordinates $X_1, \ldots, X_g$ at the origin on $J(C)$, where
$g = g(C)$ is the (arithmetic) genus, the group law will be given by a $g$-tuple $G(X,Y)$ of power series in $X = (X_1, \ldots, X_g)$, 
$Y = (Y_1, \ldots, Y_g)$ satisfying 
\begin{enumerate}
\item $G(X,Y) \equiv X + Y \mod \deg 2$,
\item $G(G(X,Y),Z) = G(X,G(Y,Z))$ where $Z = (Z_1,\ldots,Z_g)$ is another collection of indeterminates,
\item $G(X,Y) = G(Y,X)$.
\end{enumerate}
It is not difficult to show that for any such $G(X,Y)$, there is a formal inverse $i(X)$, a $g$-tuple of 
power series in $X$ such that 
$$G(X,i(X)) = G(i(X),X) = 0.$$

With this terminology, we can now state the following.

\begin{thm}  
\label{FormalAbCharp}
Let $C$ be a smooth curve and let $C_i$ be the $d \ge 3$ formal curves in $\PP$ defined as above. Let $G$ be the
$g = (d-1)(d-2)/2$-dimensional commutative formal group law over $k$ obtained as the
formal completion of the group law on the Jacobian $J(C)$ at its origin.
Then there are maps $t_i : C_i \to G$ such that 
\begin{enumerate}
\item[(a)]  for a basis $\{\Omega_j : j = 1, \ldots g\}$ of the invariant differentials on $G$, the $1$-forms $t_i^*(\Omega_j)$ are
the local expansions of linearly independent $1$-forms on $C$, and
\item[(b)] there is a (maximal) \emph{Abel relation}
\begin{equation*}
\prod_G t_i(x_i(u,v)) = 0.
\end{equation*}
\end{enumerate}
(Here $\prod_G$ means the operation of combining the $t_i(x(u,v))$ using the 
formal group law $G$.  The associativity and commutativity of $G$ means that 
this is well-defined.)
\end{thm}

\begin{proof}
The existence of such a formal Abel relation follows from general facts about Jacobians that are valid over
any algebraically closed field.  As a general reference, we suggest the chapter \cite{Milne} from the 
volume from the 1984 Storrs  conference on Faltings' proof of the Mordell conjecture.  In Milne's notation, 
given a smooth point $P \in C$, there is a canonical mapping from $C$ to its Jacobian (viewed as the group
of linear equivalence classes of divisors of degree zero):
\begin{align*}
f^P : C &\longrightarrow J(C)\\
Q &\mapsto [Q - P].
\end{align*} 
We first claim that, as in Proposition 2.2 from \cite{Milne},  $(f^P)^*$ gives an isomorphism between the invariant differentials 
on $J(C)$ and the Abelian differentials on $C$:
$$(f^P)^* : H^0(J(C), \Omega^1(J(C))) \stackrel{\sim}{\longrightarrow} H^0(C, \Omega^1(C)).$$
Proposition 2.1 of \cite{Milne} shows that there is a canonical isomorphism
between the tangent space to $J(C)$ at the origin and the cohomology
group $H^1(C,\mathcal{O}_C)$:
\begin{equation}
\label{Tang}
T_0(J(C))  \stackrel{\sim}{\longrightarrow} H^1(C, \mathcal{O}_C).
\end{equation}
Our claim is true because $J(C)$ is a group variety so there 
is an isomorphism 
$$H^0(J(C), \Omega^1(J(C))) \stackrel{\sim}{\longrightarrow} T_0(J(C))^\vee.$$ 
The dual of the isomorphism given in \eqref{Tang} is an isomorphism
$$T_0(J(C))^\vee \stackrel{\sim}{\longrightarrow} H^1(C,\mathcal{O}_C)^\vee.$$
Moreover, by Serre duality, 
$$H^1(C, \mathcal{O}_C)^\vee \stackrel{\sim}{\longrightarrow} H^0(C,\Omega^1(C)).$$
Checking the commutativity of the diagram
$$
\begin{tikzcd}
H^0(J(C), \Omega^1(J(C)))  \arrow[r, "(f^P)^*"]  \arrow[d,"\sim" {anchor=south, rotate=90}] & H^0(C,\Omega^1(C)) \\
T_0(J(C))^\vee \arrow[r,"\sim"] & H^1(C,\mathcal{O}_C)^\vee  \arrow[u, "\sim" {anchor=south,rotate=90}]
\end{tikzcd}
$$       
establishes the claim in part (a). 

The $t_i$ in the statement are the local expansions of the mappings $f^{P_i}$ at the $P_i = P_i(0,0) = (0,a_{i,0})$ in 
terms of the local coordinate $x$ on $C_i$.  Note that $t_i(0) = 0$ for all $i$.   Part (a) of the theorem follows from the first claim.
Since $D_{u,v} = \sum_i P_i(u,v)$ is linearly equivalent to $D_{0,0} = \sum_i P_i(0,0)$
for all $u,v$, it follows that 
$$\prod_G  t_i (x_i(u,v)) = 0.$$
This is part (b) of the theorem.
\end{proof}

Although we have stated this result only for smooth curves, it extends without difficulty to the case
that $C$ is a reduced plane curve of degree $d$ and $J(C)$ is the generalized Jacobian.

\section{The explicit form of formal Abel relations}

Since the theory of integration really works only in characteristic $0$, we want to understand how the $t_i(x_i(u,v))$
in formal Abel relations as in Theorem~\ref{FormalAbCharp} relate to the integrals of the Abelian differentials $\omega_j$ on $C$ 
in the equations
$$\sum_{i=1}^d \int_{P_i(0,0)}^{P_i(u,v)} \omega_j$$
in the theorem over $\CC$.

\vskip 10pt
\noindent
{\bf Notation.}
\vskip 10pt
\noindent
(a) As before, the equations of the $C_i$ will be written as 
$$y = f_i(x) = \sum_{n=0}^\infty a_{i,n} x^n.$$

\vskip 10pt
\noindent
(b)  Then by a direct computation, the coordinate $x_i(u,v)$ of the point $P_i(u,v)$ is
\begin{align}
\label{firstterms}
x_i(u,v) &= a_{i,0} u + v + a_{i,1} a_{i,0} u^2 + a_{i,1} uv + (a_{i,0}^2 a_{i,2} + a_{i,0} a_{i,1}^2) u^3 + \cdots \\
             &\quad (a_{i,0}^3 a_{i,3} + 3 a_{i,0}^2 a_{i,1}a_{i,2} + a_{i,0} a_{i,1}^3) u^4 + \cdots  \nonumber
\end{align}

\vskip 10pt
\noindent
(c)  If we assign $a_{i,n}$ the weight $n$, then the coefficient of $u^r v^\ell$ in $x_i(u,v)$ is homogeneous of 
degree $r - \ell$ and isobaric of weight $r + \ell - 1$ as a polynomial in the $a_{i,q}$ for $q \le n - 1$.
Note that factors of $a_{i,0}$ contribute to the degree but not to the weight.

\vskip 10pt
\noindent
(d)  The exact form of the terms and the numerical coefficients appearing in the coefficient of $u^n$ comes from 
the classical \emph{Lagrange reversion theorem} (see \cite[p. 132--133]{WW}) for the series expansion of a solution of a functional equation.  Here
we consider the equation
$$w = u f_i(w)$$
whose solution $w = x_i(u,0)$ gives the terms in $u$ alone in the series $x_i(u,v)$.
The Lagrange reversion formula shows that the solution $w$ is given by 
$$w = \sum_{k=1}^\infty \frac{u^k}{k!} \frac{\partial^{k-1}}{\partial x^{k-1}}\left(f_i(x)^k\right).$$
We then set $x = 0$ to find the terms in $x_i(u,0)$.  In particular, the terms
appearing in the coefficient of $u^{n+1}$ are in one-to-one correspondence with \emph{partitions} of the 
integer $n$.  Parts $j$ in the partition correspond to factors of $a_{i,j}$ for $j = 1, \ldots, n$ and the numerical
coefficient of the term with $a_{i,n}^{e_n} a_{i,n-1}^{e_{n-1}} \cdots a_{i,1}^{e_1}$ is a combination of factors from
the derivatives of the powers of $x$.
For example in the formula \eqref{firstterms} above, in the coefficient of $u^4$:
\begin{itemize}
\item The term $a_{i,0}^3 a_{i,3}$ corresponds to the partition $3 = 3$,
\item the term $3 a_{i,0}^2 a_{i,1}a_{i,2}$ corresponds to the partition $3 = 1 + 2$ 
and
\item the term $a_{i,0} a_{i,1}^3 $ corresponds to the partition $3 = 1 + 1 + 1$.
\end{itemize}
The appropriate powers of $a_{i,0}$ are included to bring the total degree to $4 = 3 + 1$.
According to the same patterns, the coefficient of the next power of $u$, namely $u^5$, will be
$$a_{i,0}^4 a_{i,4} + 4 a_{i,0}^3 a_{i,1} a_{i,3} + 2 a_{i,0}^3 a_{i,2}^2 + 6 a_{i,0}^2 a_{i,1}^2 a_{i,2} + a_{i,0} a_{i,1}^4.$$
(We note that---apart from the powers of $a_{i,0}$---there is a close similarity between this
and the form of the Fa\`a di Bruno formula for the fourth derivative of a composition.  However, the numerical coefficient 
of the $a_{i,0}^3 a_{i,2}^2$ is $2$ here and \emph{not} the $3$ that multiplies the term 
$f''(g(x)) \cdot g''(x)$ in $\frac{d^4}{dx^4}(f(g(x))$.)

\vskip 10pt
\noindent
(e)  The coefficient of $u^{m - \ell} v^\ell$ in $x_i(u,v)$ can be obtained from the coefficient of $u^m$ by applying the
formal differential operator $\frac{1}{\ell!} \frac{\partial^\ell}{\partial a_{i,0}^\ell}$.   This comes down to the binomial
expansion of $(a_{i,0} u + v)^m$ which appears repeatedly in the computation of $x_i(u,v)$. Thus, for instance, from 
the coefficient of $u^4$ given above, we find the coefficient of $u^3 v$ as 
$$3 a_{i,0}^2 a_{i,3} + 6 a_{i,0} a_{i,1} a_{i,2} + a_{i,1}^3,$$
and then the coefficient of $u^2 v^2$ is
$$3 a_{i,0} a_{i,3} + 3 a_{i,1} a_{i,2}$$
and the coefficients of $u v^3$ and $v^4$ are $a_{i,3}$ and $0$ respectively.   For future reference we note that
\begin{align}
\label{BCE}
\frac{1}{\ell!} \frac{\partial^\ell}{\partial a_{i,0}^\ell} a_{i,0}^m &= \frac{1}{\ell!} m(m-1) \cdots (m-\ell + 1) a_{i,0}^{m-\ell} \nonumber \\
 &= \binom{m}{\ell} a_{i,0}^{m - \ell}.
\end{align}
For instance, the coefficients $1,3,3,1$ of the terms in $u^{3-\ell} v^\ell$ containing the $a_{i,3}$ are the binomial coefficients in 
 $\binom{3}{\ell} a_{i,0}^{3-\ell}$.

\vskip 10pt
\noindent
The patterns in (b), (c), (e) are proved in Lemma 2.1 of \cite{Lit}.  As noted above, (d) follows from the description of 
$w = x_i(u,0) = u y_i(u,0)$ as the solution of $w = u f_i(w)$.  We note, however, that there is a 
missing term $2 a_{i,0}^3 a_{i,2}^2$ of degree $5$ and weight $4$ in the equation for the coefficient of $u^5$ in part (a) of that Lemma.  
(Fortunately), that term does not enter in a crucial way into any of the arguments presented in \cite{Lit}.

\vskip 10pt
\noindent
(f)  To illustrate some of the differences between the cases of characteristic zero and finite characteristic, we include
the following brief discussion of one of the basic features of formal group laws.

\begin{lem}
\label{FGNF1}
 Every formal group law over an algebraically closed field $k$ of characteristic $p$ is strictly isomorphic to one of the form
$$H(X,Y) \equiv X + Y \bmod \deg p.$$
If $k$ has characteristic zero, then every $m$-dimensional formal group law is strictly isomorphic to the additive $m$-dimensional 
group law
$$H(X,Y) = X + Y.$$
\end{lem}

\begin{proof}
Briefly, the idea of a naive proof is as follows. (We will use much more precise structure theorems
later.)   Saying two $m$-dimensional formal group laws $G, H$ are strictly isomorphic means that there is 
an $m$-tuple of power series $f\in k[[t_1,\ldots,t_m]]^m$ satisfying 
$$f(t_1,\ldots, t_m) \equiv (t_1, \ldots, t_m) \bmod \deg 2$$  
such that $f(G(X,Y)) = H(f(X), f(Y))$.    Taking $H(X,Y) = X + Y$, the $m$-dimensional formal
additive group law, we may seek to construct
a \emph{strict logarithm} for $G$, that is, power series 
$f \in (k[[t_1,\ldots, t_m]])^m$ such that $f(t_1, \ldots, t_m) \equiv (t_1,\ldots, t_m) \bmod \deg 2$ and 
$$f(G(X,Y)) = f(X) + f(Y).$$
It is not difficult to see that such logarithms always exist for $G$ \emph{in characteristic zero}.
For instance, if 
$$G(X,Y) = X + Y + XY = (X+1)(Y+1) - 1$$ 
is the 1-dimensional multiplicative formal group law,
then the characteristic zero logarithm is given by 
$$f(t) = \log(t+1) = t - \frac{t^2}{2} + \frac{t^3}{3} - \frac{t^4}{4} + \cdots .$$
This is true since
\begin{align*} f(G(X,Y)) &= \log((X+1)(Y+1))\\
&= \log(X+1) + \log(Y+1) = f(X) + f(Y).
\end{align*}
If $k$ has characteristic $p$, the construction of the logarithm can be carried out degree by degree
using the invariant differentials on $G$ until division by $p$ would have to
occur.  But then the truncated logarithm $\overline{f}(t)$ will yield a strict isomorphism between the
formal group law $G(X,Y)$ and a formal group law satisfying
$$\overline{f}(G(X,Y)) \equiv G(\overline{f}(X), \overline{f}(Y)) \equiv  \overline{f}(X) + \overline{f}(Y) \bmod \deg p.$$
This establishes the claim.
\end{proof}

\vskip 10pt
\noindent
(g)  Finally, in the maximal Abel relation for a curve of genus $g$, the $t_i$ will have the form
$$t_i = (t_i^{(1)} (x), \ldots, t_i^{(g)}(x)),$$
where 
$$t_i^{(j)}(x) = \sum_{n=1}^\infty A_{i,n}^{(j)} x^n \in k[[x]].$$

For the moment, we will view the $a_{i,n}$ and the $A_{i,n}^{(j)}$ as indeterminates.
Writing out the relation 
$$\prod_G t_i(x_i(u,v)) = 0$$
as in Theorem~\ref{FormalAbCharp}
and setting the coefficients of $u^m v^\ell$ equal to zero for all $m, \ell \ge 0$, we get 
a universal system of polynomial equations in the $a_{i,n}$ and the $A_{i,n}^{(j)}$ with $\ZZ$-coefficients
whose solutions describe all possible maximal Abel relations.  

\vskip 10pt
We will use the following notation.

\begin{defn}  
\label{AXDef}
{\rm For each $n \ge 2$, and letting $1\le i \le d$ in each case we will consider:
\begin{enumerate}
\item[(1)]  First, the space of maximal Abel relations modulo degree $n$ is
\begin{align*}
\mathcal{A}_n^{\rm max} &= \{ (x_i \bmod \deg n, t_i \bmod \deg n, G \bmod \deg n)\ \\
&\qquad \vert\ \prod_G t_i(x_i(u,v)) \equiv 0 \bmod \deg n \},
\end{align*}
where $G$ is a formal group law of dimension $g = (d-1)(d-2)/2$.
\item[(2)]  Disregarding the mappings $t_i$, we also let
\begin{align*}
\mathcal{X}_n^{\rm max} &= \{ (x_i \bmod \deg n)\ \vert\ \exists t_i, G \bmod \deg n \text{ such that } \\
&\qquad (x_i, t_i, G) \in \mathcal{A}_n^{\rm max}\}.
\end{align*}
\end{enumerate}
By this definition, $\mathcal{A}_n^{\rm max}$ and $\mathcal{X}_n^{\rm max}$ 
are quasi-affine schemes over ${\rm Spec}\ \ZZ$.  

As above, we note that $x_i \bmod \deg n$ uses only the coefficients from the $f_i(x) \bmod x^{n-1}$.  Hence the coordinates
of the ambient space of $\mathcal{A}_n^{\rm max}$ are the 
$$a_{i,0}, a_{i,1}, \ldots, a_{i,n-2},$$ 
the $$A_{i,1}^{(j)}, A_{i,2}^{(j)}, \ldots, A_{i,n-1}^{(j)},$$
where $1 \le i \le d$ and $1 \le j \le g = (d-1)(d-2)/2$, as well as coefficients of the formal group law $G$.  
However, by Lemma~\ref{FGNF1}, if $p > d + 2$, we can assume that the formal group law coefficients only
enter in higher degrees.
}
\end{defn}

\noindent
{\bf Note.}  We have found it more convenient to define $\mathcal{A}_n$ and $\mathcal{X}_n$ as here.  In the notation used in \cite{Lit}, the 
subscripts are shifted by $1$.  

\section{The first steps for general $d \ge 4$ for sufficiently large $p$}
\label{FirstSteps}

In this section, we will show that if $p > d + 2$, then the ``first part'' of a particular Abel relation always agrees with the Abel relation from an algebraic
curve $C$ of degree $d$.   In fact, this is a part of the argument for the characteristic $p$ converse result. The precise statement is the following.

\begin{thm}
Let $p > d + 2$, let $k$ be an algebraically closed field of characteristic $p$, 
and consider the projection $\mathcal{A}_{d+2}^{\rm max} \to \mathcal{X}_{d+2}^{\rm max}$.  If $(x_i, t_i, G) \in \mathcal{A}_{d+2}(k)$
for a particular set of mappings $t_i$ to be described below (see \eqref{Ai1s13}, \eqref{Ai2s2} below), then
the corresponding $f_i(x)$ are the local expansions of an algebraic curve of degree $d$ defined over $k$.  In other words, 
$$\mathcal{X}_{d+2}^{\rm max} =
 \{(x_i) : f_i \bmod \deg\, (d+1) \text{ from an algebraic curve of degree $d$ }\}.$$
\end{thm}

\begin{proof}
For simplicity of presentation, the details will be given for the case $d = 4$, so we assume $p \ge 7$.  The result is 
true in general, though, and may be proved by the same sort of reasoning, \emph{mutatis mutandis}.

As in Lemma~\ref{FGNF1}, we may assume that $G$
has the form $G(X,Y) \equiv X + Y \bmod \deg p$.   With this normalization, the 
form of the equations defining $\mathcal{A}_6 = \mathcal{A}_{4 + 2}$ is as follows.
From terms of degree $1$ in $u,v$:
\begin{align*}
u : &\sum_{i=1}^4 a_{i,0} A_{i,1}^{(j)} =0\\
v:  &\sum_{i=1}^4 A_{i,1}^{(j)} = 0
\end{align*}
From terms of degree $2$ in $u,v$:
\begin{align*}
u^2: &\sum_{i=1}^4 a_{i,0}^2 A_{i,2}^{(j)} + a_{i,0} a_{i,1} A_{i,1}^{(j)} = 0\\
uv: &\sum_{i=1}^4 2 a_{i,0} A_{i,2}^{(j)} + a_{i,1} A_{i,1}^{(j)} = 0\\
v^2: &\sum_{i=1}^4 A_{i,2}^{(j)} = 0
\end{align*}

As in item (d) in the Notation in the previous section, the equation from $u^{m - \ell} v^\ell$
is obtained from the equation for $u^m$ by applying the formal differential operator
$$\frac{1}{\ell!}\left(\frac{\partial^\ell}{\partial a_{1,0}^\ell} +  \frac{\partial^\ell}{\partial a_{2,0}^\ell} + \frac{\partial^\ell}{\partial a_{3,0}^\ell} + \frac{\partial^\ell}{\partial a_{4,0}^\ell}\right).$$
Thus the remaining equations can be obtained from the following:  from degree $3$, 
$$u^3:  \sum_{i=1}^4 a_{i,0}^3 A_{i,3} + 2 a_{i,0}^2 a_{i,1}  A_{i,2} + (a_{i,0}^2 a_{i,2} + a_{i,0} a_{i,1}^2) A_{i,1}  = 0.$$
Then from degree $4$:
\begin{align*}
u^4:  &\sum_{i=1}^4 a_{i,0}^4 A_{i,4}^{(j)} + \sum_{i=1}^4 3 a_{i,0}^3 a_{i,1} A_{i,3}^{(j)} + \sum_{i=1}^4 (2 a_{i,0}^3 a_{i, 2} + 3 a_{1, 0}^2 a_{i,1}^2) A_{i,2}^{(j)}\\
        & + (a_{i,0}^3 a_{i,3} + 3 a_{i,0}^2 a_{i,1}a_{i,2} + a_{i,0} a_{i,1}^3) A_{i,1}^{(j)} = 0.
 \end{align*}
Finally, from degrees $\ge 5$:
\begin{align*}
u^5: 0 &= \sum_{i=1}^4 a_{i,0}^5 A_{i,5}^{(j)} + \sum_{i=1}^4 4 a_{i,0}^4 a_{i,1} A_{i,4}^{(j)} + \sum_{i=1}^4 (3 a_{i,0}^4 a_{i,2} + 6 a_{i,0}^2 a_{i,1}^2) A_{i,3}^{(j)} \\
           &\qquad + \cdots  + \sum_{i=1}^4 (a_{i,0}^4 a_{i,4} + 4 a_{i,0}^3 a_{i,1} a_{i,3} + \cdots + a_{i,0} a_{i,1}^4) A_{i,1}^{(j)},
\end{align*}
and generally
\begin{align}
\label{HigherDegrees}
u^n: 0&= \sum_{i=1}^4 a_{i,0}^n A_{i,n}^{(j)} + \sum_{i=1}^4 (n-1) a_{i,0}^{n-1} a_{i,1} A_{i,n-1}^{(j)}\nonumber \\ 
& \quad + \sum_{i=1}^4 \left((n-2) a_{i,0}^{n-1} a_{i,2} + \binom{n-1}{2}  a_{i,0}^{n-2} a_{i,1}^2\right) A_{i,n-2}^{(j)} + \cdots  \\
& \quad + (a_{i,0}^{n-1} a_{i,n-1} + (n-1) a_{i,0}^{n-2} a_{i,1} a_{i,n-2 }+ \cdots  + a_{i,0} a_{i,1}^{n-1}) A_{i,1}^{(j)}.\nonumber
\end{align}

Much of the information we need is actually contained in these terms containing $A_{i,n}^{(j)}$, $A_{i,n-1}^{(j)}$ and $A_{i,1}^{(j)}$ from the terms of total degree $n$ in $u,v$.  

We begin with some observations.  We want to consider Abel relations in a $g$-dimensional formal group law with $g = (4 - 1)(4 - 2)/2 = 3$, 
so we take $1 \le j \le 3$ here.  However the form of the equations from the coefficients of $u,v, u^2, uv, v^2$ shows that there are 
$2$ degrees of freedom in the choice of the $A_{i,1}^{(j)}$ and $1$ further degree of freedom in the choice of the $A_{i,2}^{(j)}$.  Then from the terms of degree $m = 3$ in $u,v$ on
to degree $m = p - 1$ in the truncated Abel relation, the $A_{i,m}^{(j)}$  are uniquely determined by the earlier choices and the coefficients of the $C_i$.  This follows because for
each $j$ the 
coefficient matrix of the system for the $A_{i,m}^{(j)}$ have the following form.  For $m = 3$, for instance, the matrix is
$$\begin{pmatrix}
a_{1,0}^3 & a_{2,0}^3 & a_{3,0}^3 & a_{4,0}^3\\
3 a_{1,0}^2 & 3 a_{2,0}^2 & 3 a_{3,0}^2 & 3 a_{4,0}^2\\
3 a_{1,0} & 3 a_{2,0} & 3 a_{3,0} & 3 a_{4,0}\\
1 & 1 & 1 & 1
\end{pmatrix}$$
whose determinant is $9$ times the usual \emph{Vandermonde determinant} of the $a_{i,0}$.  For $m > 3$, the coefficient matrices of the 
$A_{i,\ell}^{(j)}$ terms have a similar Vandermonde form, but with more rows and numerical coefficients form the binomial coefficients $\binom{m}{\ell}$.  
We assume the $a_{i,0}$ are distinct, so the determinant of the $4 \times 4$ submatrix in the last four rows of these matrices is nonzero.  
That implies the $A_{i,m}^{(j)}$ for $m \ge 3$ are uniquely determined by the $A_{i,1}^{(j)}, A_{i,2}^{(j)}$ and the coefficients of the $C_i$.  
The upshot is that  there is a $3$-dimensional space of solutions for the $A_{i,1}$ and $A_{i,2}$.  
%This is what we meant earlier about why it might be more natural to assume the formal group law is 3-dimensional for the case $d = 4$.  

However, and this will be a key observation, the coefficients $A_{i,n}^{(j)}$ for $n \ge 3$ are always determined by exactly the same
equations in terms of the $A_{i,1}^{(j)}$, the $A_{i,2}^{(j)}$, and the coefficients $a_{i,n}$ of the equations of the $C_i$.  
In the following, it will be most convenient to make a particular choice of coordinates in $G$, so to speak, by picking one particular set of 
solutions of the equations for the $A_{i,1}^{(j)}$ and $A_{i,2}^{(j)}$.  
We will use the following notation:  For each $1 \le i \le 4$, let 
\begin{equation}
\label{deltai}
\delta_i = \prod_{j\ne i} (a_{i,0} - a_{j,0}).
\end{equation}
Then it is easy to see that the following choices give two linearly independent solutions for the equations from $u,v$ above.
\begin{equation}
\begin{matrix}
\label{Ai1s13}
A_{1,1}^{(1)} = 1/\delta_1 & A_{2,1}^{(1)}  = 1/\delta_2 &  A_{3,1}^{(1)}  = 1/\delta_3 & A_{4,1}^{(1)}  = 1/\delta_4.\\
A_{1,1}^{(3)}  = a_{1,0}/\delta_1 & A_{2,1}^{(3)} = a_{2,0}/\delta_2 &  A_{3,1}^{(3)} = a_{3,0}/\delta_3 & A_{4,1}^{(3)} = a_{4,0}/\delta_4\\
\end{matrix}
\end{equation}

Substituting these into the equations from $u^2, uv, v^2$ we can solve for the $A_{i,2}^{(1)}$, $A_{i,2}^{(3)}$.
On the other hand, with the choice $A_{i,1}^{(2)} = 0$, $1 \le i \le 4$, we have 
\begin{equation}
\begin{matrix}
\label{Ai2s2}
A_{1,2}^{(2)} = 1/(2\delta_1) & A_{2,2}^{(2)} = 1/(2\delta_2) & A_{3,2}^{(2)} = 1/(2\delta_3) & A_{4,2}^{(2)} = 1/(2\delta_4).
\end{matrix}
\end{equation}
These agree with the first terms in the truncated local expansions of the formal integrals of the differentials
\begin{equation}
\label{AbInts}
  \omega_1 = \frac{dx}{\partial f/ \partial y},\  \omega_2 = \frac{x\, dx}{\partial f/\partial y},\ \omega_3 = \frac{y\, dx}{\partial f/\partial y},\
\end{equation}
where we continue the expansion until computing the integral requires inversion of an exponent containing the prime $p$.
In other words, it is not difficult to check that in fact we can take  sets of coefficients for $j = 1,2,3$ to
 agree with the terms $\bmod\ x^p$ in the series expansions of the integrals of the Abelian differentials 
on an algebraic quartic curve in characteristic zero.  
%We will now indicate two distinct (but related) ways to establish the result of this Theorem.  As mentioned above, we only show the 
%details for $d = 4$, but the ideas extend to all $d \ge 4$ in a relatively straightforward way.  
%For the first approach, we begin by noting that the $a_{i,0}$ and $a_{i,1}$ from the $C_i$ are completely arbitrary 
%(although we require as always that the $a_{i,0}$ are distinct.  Since the $y = a_{i,0} + a_{i,1} x$ are always the 
%local expansions of algebraic quartics (e.g. the reducible quartic made up of the lines with those $4$
%equations), it follows that any $x_i \bmod\ x^2$ and the  $t_i \bmod\ x^3$ can appear in truncated Abel relations
%$\bmod \deg 3$ and furthermore the possible $A_{i,1}^{(j)}$ and $A_{i,2}^{(j)}$ always form $3$-dimensional vector
%space determined by the $a_{i,0}$ and $a_{i,1}$.  
%Furthermore, it is also true that if we consider $C_i \bmod\ x^3$ from an algebraic quartic curve, for which the 
%first formal Reiss relation
%$$a_{1,2} + a_{2,2} + a_{3,2} + a_{4,2} = 0$$
%is satisfied, there is a truncated Abel relation using $A_{i,1}^{(j)}$ and $A_{i,2}^{(j)}$ as above where the $A_{i,3}^{(j)}$
%come from the expansions $\mod \deg 4$ of the integrals in \eqref{AbIints}.  
More precisely,  letting $f(x,y) = \prod_{i=1}^4 (y - f_i(x))$, we can take
%\begin{align*}
$$t_i^{(1)}(x) \equiv  \int \left. \frac{dx}{\partial f/\partial y}\right|_{C_i}  \bmod x^p, \quad
t_i^{(2)}(x) \equiv \int \left. \frac{x\, dx}{\partial f/\partial y}\right|_{C_i} \bmod x^p, $$
and
$$t_i^{(3)}(x) \equiv \int \left. \frac{y\, dx}{\partial f/\partial y}\right|_{C_i} = \int \left. \frac{f_i(x) dx}{\partial f/\partial y}\right|_{C_i} \bmod x^p.$$
%\end{align*}
These make sense since we assume $p \ge 7$ so we can still
invert the necessary exponents to get truncated expansions for the integrals.  

The $(x_i) \in \mathcal{X}_6$ from algebraic quartic curves come from a smooth, $14$-dimensional 
subvariety $\mathcal{C}_6 \subset \mathcal{X}_6$.  (Note that there is also a $\binom{4 + 2}{2} = 15$-dimensional vector space
of homogeneous polynomials of degree $4$ in $x,y$, hence a $14$-dimensional projective space of all plane quartics.)

The following fact is obvious, but it is important for our approach.
\vskip 10pt
\noindent
{\it Claim.}
The $(x_i)$ in $\mathcal{C}_6$ are the local expansions 
for which
\begin{equation}
\label{prodexp}
\prod_{i=1}^4 (y - f_i(x) \bmod x^5)
\end{equation}
has total degree $4$ in $x,y$.

\vskip 10pt
This leads to a set of equations on the coefficients in the $f_i(x)$ of the following form.  First from the coefficients of $y^3 x^2$, $y^3 x^3$, and $y^2 x^3$ in
the expansion of \eqref{prodexp},
\begin{align}
\label{Reiss4a}
0 &= \sum_{i=1}^4 a_{i,2} \nonumber \\
0 &= \sum_{i=1}^4 a_{i,3}  \\
0 &=\sum_{i=1}^4 \left(\left(\sum_{j \ne i} a_{j,0}\right) a_{i,3} + \left(\sum_{j\ne i} a_{j,1}\right) a_{i,2}\right) = 0. \nonumber
\end{align}
 In addition, from the coefficients of $y^3 x^4$, $y^2 x^4$, and $y x^4$ in \eqref{prodexp}, the following additional relations must hold on the coefficients of $x^4$:
\begin{align}
\label{Reiss4b}
0 &= \sum_{i=1}^4 a_{i,4} \nonumber \\
0 &= \sum_{i=1}^4 \left( \left(\sum_{j \ne i} a_{j,0}\right) a_{i,4} + \left(\sum_{j\ne i} a_{j,1}\right) a_{i,3} \right) + \sum_{1 \le i < j \le 4} a_{i,2} a_{j,2}  \\
0 &=\sum_{i=1}^4 \left(\left(\sum_{\substack{1 \le j < k \le 4\\  j,k \ne i}} a_{j,0} a_{k,0}\right) a_{i,4} + \sum_{i=1}^4 \left(\sum_{\substack{1 \le j,  k \le 4,\\  j,k \ne i}} a_{j,0} a_{k,1}\right) a_{i,3} \right) \nonumber\\ 
&\qquad + \sum_{i=1}^4 \left(\sum_{\substack{1 \le j < k \le 4,\\ j,k \ne i}} a_{j,1} a_{k,1} \right) a_{i,2} + \sum_{\substack{1 \le i < j \le 4,\\ \{k,\ell\} \cup \{i,j\} = \{1,2,3,4\} }} (a_{k,0} + a_{\ell,0}) a_{i,2} a_{j,2} \nonumber
\end{align}
Under the assumption that the $a_{i,0}$ are distinct, it is easy to check
that these equations define a smooth 14-dimensional subvariety 
$$\mathcal{C}_6 \subset (\A^{20})^* = {\rm Spec}\ \ZZ[a_{i,0}, \ldots, a_{i,4},  (a_{i,0} - a_{j,0})^{-1}, 1 \le i < j \le 4].$$  
We claim that in fact $\mathcal{X}_6$ equals this $\mathcal{C}_6$. 

A direct computation (we used the Maple computer algebra system for this) shows that substituting the corresponding values for the $A_{i,n}^{(1)}$, 
$n = 1, 2,3$ from 
$$t_i^{(3)}(x) =  \int \left. \frac{y\, dx}{\partial f/\partial y}\right|_{C_i} = \int \left. \frac{f_i(x) dx}{\partial f/\partial y}\right|_{C_i} \bmod x^6,$$
and simplifying, the equation from $u^3$ for $j = 3$ implies
$$0 = a_{1,2} + a_{2,2} + a_{3,2} + a_{4,2},$$
the first equation from \eqref{Reiss4a}.  (The same relation can be obtained in other ways as well from the equations using the $t_i^{(j)}$ for $j = 1,2$.)  

Similarly, the equation from $u^3 v$ for $j=3$ implies
$$0 = a_{1,3} + a_{2,3} + a_{3,3} + a_{4,3},$$
which is the second equation from \eqref{Reiss4a}.
The equation from $u^4$ for $j = 3$ implies
$$0 = \sum_{i=1}^4 (2 a_{i,0} + \sum_{j \ne i} a_{j,0}) a_{i,3} + \sum_{i=1}^4 (2 a_{i,1} + \sum_{j \ne i} a_{j,1}) a_{i,2}$$
Using the previous two equations, this clearly reduces to the form of the third equation in \eqref{Reiss4a}.  
The equation from $u^3 v^2$ for $j = 3$ implies
$$0 = a_{1,4} + a_{2,4} + a_{3,4} + a_{4,4},$$
which is the first equation from \eqref{Reiss4b}.   The equations from $u^5$ and $u^4 v$ for $j = 1$ imply other
relations on the $a_{i,n}$ with $1 \le n \le 4$.  It can be checked by a Gr\"obner basis computation that the ideal generated by the coefficients of $u^{m - \ell} v^\ell$
for all $m = 1,2,3,4,5$ is equal to the ideal generated by the six right-hand sides in \eqref{Reiss4a} and \eqref{Reiss4b}.  Hence every element of 
$\mathcal{X}_6$ comes from $C_i$ where the coefficients satisfy \eqref{Reiss4a} and \eqref{Reiss4b}, and hence agree with the local expansions
of an algebraic quartic $\bmod\ x^5$.  
\end{proof}

The first equation in \eqref{Reiss4a} is a form of the so-called \emph{Reiss relation} for the intersections with the line $\Delta_{0,0}$ (see \cite[p. 677]{GH}) and 
the other equations also follow from the Reiss relations for the intersections with the lines $\Delta_{u,v}$.  

\section{The terms of degree $p$ in a maximal Abel relation}

As before, 
we will consider only the case $d = 4$ explicitly.
By Lemma~\ref{FGNF1},
the terms of degree $p$ are the first ones where the form of the formal group law of $J(C)$ must be taken into account.  We have the following
precise statement.  Let 
$$C_p(x,y) = \frac{-(x+y)^p + x^p + y^p}{p} \mod p \quad  \in \FF_p[x,y].$$
(That is, the power is expanded first in characteristic zero, all factors of $p$ from the binomial coefficients are cancelled, and then the coefficients 
are reduced mod $p$.)  For example
$$C_7(x,y) = 6 x^6 y +  4 x^5 y^2 + 2 x^4 y^3 + 2 x^3 y^4 + 4 x^2 y^5 + 6 x y^6 \in \FF_7[x,y].$$

\begin{thm}  Every formal group law of dimension $g$ over an algebraically closed field $k$ of characteristic $p$ is strictly isomorphic
\label{FGLawForm}
to one of the form 
\begin{equation}
\label{FGL}
G(X,Y) \equiv X + Y + \Gamma\, C_p(X,Y) \bmod \deg (p+1),
\end{equation}
where $\Gamma = (\gamma_{i,j})$ is some $g\times g$ matrix with entries in $k$ and $C_p(X,Y)$ is the $g$-component column vector of polynomials $C_p(X_i,Y_i)$
for $i = 1, \ldots, g$.
\end{thm}

This is the result of \cite[Prop. 20.1.7]{H}.  

\vskip 10pt

The matrix $\Gamma$ has several more conceptual interpretations.  The first makes use of an operation
on differential forms in characteristic $p$ that reflects an important difference with the characteristic zero situation (see, for example, \cite{Serre}).  
Let $k$ be 
a perfect field of characteristic $p$.  Then the polynomials $1, t, \ldots, t^{p-1}$ give a $p$-basis for 
$k[[t]]$ over $k[[t^p]]$.  That is, every $f(t) = \sum_{n=0}^\infty a_n t^n \in k[[t]]$ can be written uniquely as
$$f(t) = \sum_{j=0}^{p-1} f_j^p(t) t^j.$$
If $\omega = f(t)\, dt$, then the \emph{Cartier operator} $\mathcal{C}$ is defined on $\omega$ by setting
$$\mathcal{C}(\omega) = f_{p-1}\, dt$$
It can be shown that $\mathcal{C}$ is well-defined and gives the same result under any ``reparametrization''
$t \mapsto a_1 t + a_2 t^2 + \cdots \in k[[t]]$ with $a_1 \ne 0$.   The Cartier operator satisfies the following properties:
\vskip 5pt
\begin{enumerate}
\item[(a)] $\mathcal{C}(\omega_1 + \omega_2) = \mathcal{C}(\omega_1) + \mathcal{C}(\omega_2)$. \vskip 2pt
\item[(b)] $\mathcal{C}(g^p \omega) = g \mathcal{C}(\omega)$ for all $g \in k[[t]]$.\vskip 2pt
\item[(c)] $\mathcal{C}(df) = 0$ if and only if $\omega$ is an exact form: $\omega = df$ for some $f$.  (Note
this is not automatic in characteristic $p$.  In intuitive terms, computing a formal integral for $\omega = f(t)\ dt$ where
$f(t) \in k[[t]]$ might require inversion of exponents
divisible by $p$.  We will come back to this point later in some examples.) \vskip 2pt
\item[(d)] $\mathcal{C}\left(\frac{df}{f}\right) = \frac{df}{f}$ for all ``logarithmic'' differentials $\frac{df}{f}$.
(We note that since $\mathcal{C}$ is not linear over the algebraic closure of $\FF_p$ by (b), this \emph{does not}
say that the logarithmic differentials are eigenvectors with eigenvalue $\lambda = 1$, as one might be
tempted to believe.)
\end{enumerate}

Similarly, for differential forms $\omega = f(t_1,\ldots,t_m)\, dt_1 \wedge \cdots \wedge dt_m$ in several variables,
there are unique $f_{c_1, \ldots, c_m}$ for $0 \le c_i \le p-1$ such that 
$$\omega = \sum_{c_1,\ldots, c_m} f_{c_1,\ldots,c_m}^p t_1^{c_1} \cdots t_m^{c_m}\, dt_1 \wedge \cdots \wedge dt_m$$
and then the Cartier operator is defined by setting 
$$\mathcal{C}(\omega) = f_{p-1,\ldots, p-1}\, dt_1 \wedge \cdots \wedge dt_m.$$
The algebraic properties are the same as for the one-variable version.

Our first promised conceptual interpretation for the matrix $\Gamma$ in \eqref{FGL} come from the following connection.

\begin{prop}
\label{CartierFGL}
The matrix $\,\Gamma$ shows the results of applying the Cartier operator to the 
invariant differentials on the formal group law $G$.
\end{prop}

\begin{proof}
For simplicity, we will show the details for the case of a $3$-dimensional formal group law because we will 
use that case to understand Abel relations on quartic curves where $g = (4-1)(4-2)/2 = 3$.  From 
Equation~\eqref{FGL}, modulo terms of degree $p+1$ and higher, we have   $G(X,Y) \equiv $
\begin{align}
\label{fgrouplaw}
&(X_1+Y_1 + \gamma_{1,1} C_p(X_1,Y_1) + \gamma_{1,2} C_p(X_2,Y_2) + \gamma_{1,3}C_p(X_3,Y_3),  \nonumber \\
&\ X_2+Y_2 + \gamma_{2,1} C_p(X_1,Y_1) + \gamma_{2,2} C_p(X_2,Y_2) + \gamma_{2,3}C_p(X_3,Y_3), \\
&\ X_3+Y_3 + \gamma_{3,1} C_p(X_1,Y_1) + \gamma_{3,2} C_p(X_2,Y_2) + \gamma_{3,3}C_p(X_3,Y_3) ). \nonumber
\end{align}
Letting $D = DG_X(0,Y)$, that is, the Jacobian matrix of derivatives with respect to the $X_i$ evaluated at $(0,Y)$, we find
$$D \equiv \begin{pmatrix} 1 + \gamma_{1,1} Y_1^{p-1} & \gamma_{1,2} Y_2^{p-1} & \gamma_{1,3} Y_3^{p-1} \\
                                     \gamma_{2,1} Y_1^{p-1} & 1 + \gamma_{2,2} Y_2^{p-1} & \gamma_{2,3} Y_3^{p-1}\\
                                     \gamma_{3,1} Y_1^{p-1} & \gamma_{3,2} Y_2^{p-1} & 1 +  \gamma_{3,3} Y_3^{p-1} \end{pmatrix} \bmod \deg p.$$
and by \cite[p. 243]{Freije}, the invariant differentials are obtained from the rows of the inverse matrix $D^{-1}$ as follows.  One
basis is
\begin{align}
\label{InvDiffsJ}
\Omega_1 &\equiv (1 + \gamma_{1,1}Y_1^{p-1}) dY_1 + \gamma_{1,2} Y_2^{p-1} dY_2 + \gamma_{1,3} Y_3^{p-1} dY_3  \bmod \deg p\nonumber \\
\Omega_2 &\equiv   \gamma_{2,1}Y_1^{p-1} dY_1 + (1 + \gamma_{2,2} Y_2^{p-1}) dY_2 + \gamma_{2,3} Y_3^{p-1} dY_3  \bmod \deg p\\
\Omega_3 &\equiv   \gamma_{3,1}Y_1^{p-1} dY_1 + \gamma_{3,2} Y_2^{p-1} dY_2 + (1 + \gamma_{3,3} Y_3^{p-1}) dY_3  \bmod \deg p. \nonumber
\end{align} 
Hence applying the Cartier operator and expanding the results in terms of the $\Omega_j$, we have 
\begin{align}
\label{CInvDIffsJ}
\mathcal{C}(\Omega_1) &= \gamma_{1,1} dY_1 + \gamma_{1,2} dY_2+ \gamma_{1,3} dY_3 + \cdots \nonumber\\
&= \gamma_{1,1} \Omega_1 + \gamma_{1,2}\Omega_2 + \gamma_{1,3} \Omega_3\nonumber \\
\mathcal{C}(\Omega_1) &= \gamma_{2,1} dY_1 + \gamma_{2,2} dY_2 + \gamma_{2,3} dY_3 + \cdots\\
&= \gamma_{2,1} \Omega_1 + \gamma_{2,2}\Omega_2 + \gamma_{2,3} \Omega_3\nonumber\\
\mathcal{C}(\Omega_1) &= \gamma_{3,1} dY_1+ \gamma_{3,2} dY_2+  \gamma_{3,3} dY_3 + \cdots\nonumber\\
&=  \gamma_{3,1} \Omega_1 + \gamma_{3,2}\Omega_2 + \gamma_{3,3} \Omega_3.\nonumber
\end{align}
In other words, the entries of the matrix $\Gamma$ give the expansions of the Cartier operator
on the invariant differentials on the formal group law.
 \end{proof}             
 
 Because of this result, combined with the facts used in the proof of Theorem~\ref{FormalAbCharp}, we get a close
 relation between this matrix $\Gamma$ in the formal group law of the Jacobian of a curve and the so-called
 \emph{Cartier-Manin matrix} of the curve, namely the matrix expressing the $\mathcal{C}(\omega_j)$, $j = 1, \ldots, g$ 
 in terms of the $\omega_j$ for a basis $\{\omega_1, \ldots, \omega_g\}$ of $H^0(C, \Omega^1(C))$.
 This matrix is also closely related to (but dual to and hence not exactly the same as) the so-called \emph{Hasse-Witt matrix} of $C$.
 See \cite{AchterHowe} for a discussion of the relation and a warning about consequences of some past
 confusions between the Cartier-Manin and Hasse-Witt matrices in the literature.
 
 A second more conceptual interpretation for the meaning of the matrix $\Gamma$ comes from the following
 fact.
 
 \begin{prop}
 \label{pEndo}
 Over a field of characteristic $p$, the endomorphism $[p]$ of the formal group law from Equation \eqref{FGL} is given by 
 $$[p](X) = \Gamma X^p \bmod \deg\, (p + 1),$$
where $X^p$ is the vector of $p$-th powers $X_i^p$.
\end{prop}

\begin{proof}
Here $[p](X)$ means the series
$$[p](X) = G(X,G(X,G(X, \cdots )))$$
nesting to level $p-1$.  This can be verified by a direct computation from \eqref{FGL}.  It also follows from various 
relations in the so-called Dieudonn\'e module of the formal group law.  See \cite{H}.
\end{proof}

For future reference we will conclude this section by describing a coarser classification result for formal group laws
over algebraically closed fields of characteristic $p$, often called the \emph{Dieudonn\'e-Manin classification} (see \cite{O}).  
This is coarser because it classifies formal group laws up to an equivalence relation 
known as \emph{isogeny} rather than \emph{strict isomorphism} as above.   Two formal group laws are said
to be \emph{isogenous} if their Newton polygons are equal.  A different, but equivalent, definition in terms
of formal power series is given in \cite[(28.3.4)]{H}.  

\begin{thm}  
\label{DieudonneManin}
Over an algebraically closed field $k$ of characteristic $p$, up to isogeny,
every finite-dimensional formal group law $G$ can be written as a sum of formal group laws of the form
$$G \sim \sum_{i} G_{m_i, n_i},$$
(that is, the variables $X$, $Y$ and the vector of formal power series $G(X,Y)$ have a product decomposition
of this shape) where 
\begin{itemize}
\item All the group laws $G_{m_i,n_i}$ are defined over the prime field, $\FF_p$.
\item $\gcd(m_i, n_i) = 1$ for all $i$.
\item $\dim G_{m_i,n_i} = m_i$ for all $i$.
\item All $G_{m_i, n_i}$ are indecomposable up to isogeny.
\item The height of $G(m_i, n_i)$ (that is, the largest power $p^h$ such that $[p]_{G_{m_i,n_i}}(X) = f(X^{p^h})$ for some $f$) is $m_i + n_i$.
\item The covariant Dieudonn\'e module of the formal group law $G_{m,n}$ equals
$$D(G_{m,n}) =  W(k)[F,V]/W(k)\cdot (F^m - V^n),$$
where $W(k)$ is the ring of Witt vectors over $k$, and $F$ and $V$ are operators known as
\emph{Frobenius} and \emph{Verschiebung}, respectively.
\end{itemize}
\end{thm}

If $G$ is the formal completion of the Jacobian of a smooth curve of genus $g$, for instance, 
its Newton polygon starts at $(0,0)$ and ends at $(2g,g)$.  The polygon consists of a collection of 
line segments with integer endpoints and slopes $0 \le \lambda \le 1$ and there is an additional 
symmetry property in that case.  We will give several computations of Newton polygons
of formal Jacobians of curves in the next section, making use of the strong connection between the 
zeta-function of a curve over a finite field and the zeta-function of its Jacobian.

The fact that all of the $G_{m_i,n_i}$ are defined over $\FF_p$ can be seen from 
Hazewinkel's construction via universal formal group laws in \cite[(28.5.9)]{H}.
 
 \section{Examples of maximal formal Abel relations}
 
 In this paper, we want to illustrate how the Abel relations work for 
 various algebraic curves and the role of the terms from the formal group law.
 Hence we will consider the explicit form of the Abel relations on several particular plane
 curves of degree $4$.  We first consider smooth plane quartics (hence genus $g = 3$) and we 
 will assume the curves are defined over a finite field $\FF_q$.   To begin, here is some background information 
 about the possible formal group structures of the Jacobians of smooth curves of genus $g = 3$.  
 
 As discussed in \cite{AchterHowe}, the invariants we will compute are 
closely connected with the zeta function of $C$ over $\FF_q$.  
For small $q$, the zeta function can be computed directly and naively by counting points on the projective closure
of the curve over $\FF_q$, $\FF_{q^2}$, and $\FF_{q^3}$.)   In particular, we will take $q = 7$ in the following.
Say the number of $\FF_{7^i}$-rational points is $N_i$ for $i = 1,2,3$.
 The zeta function has the form 
$$Z(X,t) = \frac{343 t^6 + 49 a_1 t^5 + 7 a_2 t^4 + a_3 t^3 + a_2 t^2 + a_1 t + 1}{(1 - t)(1 - 7 t)}.$$
where 
\begin{align*}
a_1 &= N_1 - 7 - 1,\\
a_2 &= \frac{a_1^2 + N_2 - 49 - 1}{2},\\
a_3 &= \frac{a_1 a_2 + a_1 (N_2 - 49 - 1) + N_3 - 343 - 1}{3}.
\end{align*}
As is known from results of Weil (results that predated and informed the more general Weil conjectures), 
these follow from the factorization of the numerator
$$\sum_{k = 0}^6 c_k t^k = 343 t^6 + 49 a_1 t^5 + 7 a_2 t^4 + a_3 t^3 + a_2 t^2 + a_1 t + 1 = \prod_{j = 1}^6(1 - r_j t)$$
in $\CC[t]$, where the reciprocal roots $r_j$ are algebraic integers of absolute value $\sqrt{7}$.  
The formulas for the $a_i$ then follow from the Newton formulas expressing power sums of 
the $r_j$ in terms of the elementary symmetric polynomials in the $r_j$.  

There are several ways to describe the possible formal structures of the Jacobians of smooth genus-3
curves, up to isomorphism (over the algebraic closure of the finite field).  As above, we specialize to the case $q = 7$ to 
relate this discussion to the data from the examples to follow.  Perhaps the most satisfactory way
to package the following is via the structure of the \emph{Dieudonn\'e module} of the formal
group law.  This leads to a list of $8$ possible \emph{Ekedahl-Oort types} for the formal Jacobian of a smooth genus 3 curve.
For a convenient summary of this language and the resulting classification, consult \cite{AP}.
 However, we will only provide the Ekedahl-Oort types of our examples without detailed justification,
  because that approach is more distant from the actual
constructions involved in understanding the formal Abel relation than other ways of describing the structure.  

Instead, we will start by noting that other ways for describing the formal group structure
 can be phrased in terms of the $7$-adic \emph{Newton polygon} of the numerator polynomial
$\sum_{k=0}^6 c_k t^k$
of the zeta function.  This is the lower convex hull of the points $(k,{\rm ord}_7(c_k))$, $k = 0,1,\ldots,6$, 
a collection of line segments starting at $(0,0)$ and ending at $(6,3)$.  All of the segments have slopes $m$ satisfying 
$0 \le m \le 1$.  Because of the functional equation
of the zeta function, this collection has a symmetry property which comes down to saying that the 
number of segments of slope $m$ is the same as the number of segments of slope $1 - m$.  
The number of pairs of segments of slope $0$ and $1$ computes the so-called \emph{$p$-rank} of the Jacobian, denoted
by $f$.   This is also the dimension of the vector subspace of $H^0(C, \Omega^1(C))$ spanned by \emph{logarithmic forms} (defined over the algebraic
closure $\overline{\FF_7}$) on the curve or the formal Jacobian.  The segments of slope $\frac{1}{2}$ compute the so-called \emph{$a$-number}, which coincides with 
the dimension of the subspace of $H^0(C,\Omega^1(C))$ spanned by exact forms.  By the general facts about the Cartier operator
cited earlier, the $a$-number is also equal to the dimension of the kernel of the Cartier-Manin matrix.  It is known in general that $0 < a + f \le g$.
 
 The following examples of smooth curves were found
 by a random search in the space of plane quartics, making use of the normalization
 $$f(0,y) = y^4 + 3y^3 + y + 3 = (y+1)(y+2)(y+3)(y+4)$$
 so the curve is in a suitable position relative to the $y$-axis. 
  In other words, we did not use any of the existence results
 for such curves that have been established and appear in the literature.
 We will deal with two ``boundary cases" first---these are, in a sense, the most general kind of example
 and the least general kind of example.  
 
 First we will discuss an ``ordinary''  curve where the 
 Cartier-Manin matrix is \emph{invertible}.  This is the generic case in a sense;  it is equivalent to saying the 
 kernel of the endomorphism $[p]$ on $J(C)$ has order $p^g = 7^3$.   It also implies that the formal
 group of the Jacobian is isomorphic over the algebraic closure of $\FF_7$ to a product of three copies of
 the formal multiplicative group.   One such curve is 
 \begin{align*}
 0 = f(x,y) &= y^4 + (2 x + 3) y^3 + (2x^2 + 5x) y^2 + (5 x^3 + 4x^2 + 5x + 1)y\\
                &\qquad + 6x^4 + 6x^3 + 2x^2 + x + 3.
 \end{align*}
 
 Using the monomial basis 
 \begin{equation}
 \label{AbDiffBasis}
 \omega_1 = \frac{dx}{\partial f/\partial y}, \qquad 
 \omega_2 = \frac{x\, dx}{\partial f/\partial y}, \qquad 
  \omega_3 = \frac{y\, dx}{\partial f/\partial y}
  \end{equation}
 for $H^0(C,\Omega^1(C))$, we compute the Cartier-Manin matrix $\mathcal{C}$ using the formula from 
 \cite{StohrVoloch}.  The result is
 $$\mathcal{C} = \begin{pmatrix}   4 & 1 & 2\\
           0 & 5 & 6\\
           0 & 2 & 3\end{pmatrix},$$
which has nonzero determinant in $\FF_7$.   Hence the Jacobian of this $C$ is ordinary.  
\begin{figure}[h]
\begin{center}
\includegraphics[height=2in,width=4in]{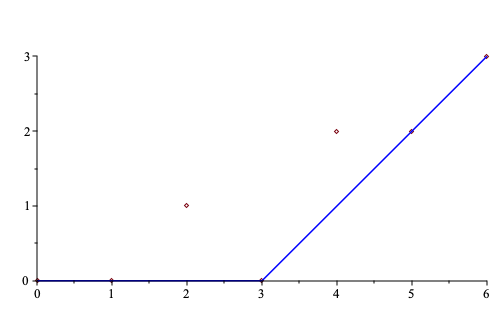}
\caption{The Newton polygon of this ordinary curve.}
\end{center}
\end{figure}
By counting points as described above we find the zeta function
$$Z(X,t) = \frac{343 t^6 + 98 t^5 + 49 t^4 + 23 t^3 + 7 t^2 + 2 t + 1}{(1 - t)(1 - 7 t)},$$
so the Newton polygon over $\FF_7$ comes from the points
$$(0,0), (1,0), (2,1), (3,0), (4,2),  (5,2), (6,3).$$
The points $(2,1)$ and $(5,2)$ lie respectively above the line from $(0,0)$ to $(3,0)$ and the line from $(3,0)$ to $(6,3)$.
Hence the slopes of the Newton polygon are only $0$ and $1$.  (I think) the fact that the line from $(1,0)$ to $(3,0)$
contains no point from the coefficients of the zeta function just reflects the fact that we will need to go to extension
fields to find the corresponding logarithmic forms.  The corresponding Ekedahl-Oort type (see \cite{AP}) would 
be the vector $(1,2,3)$, with $f = 3$ and $a = 0$.  

Using the ideas from the proof of Theorem~\ref{FormalAbCharp}---in particular the isomorphisms
$$(f^P)^* : H^0(J(C), \Omega^1(J(C))) \longrightarrow H^0(C,\Omega^1(C)),$$
and the ``monomial'' basis 
$$\frac{dx}{\partial f/\partial y}, \quad \frac{x dx}{\partial f/\partial y}, \quad \frac{y dx}{\partial f/\partial y}$$
for $H^0(C, \Omega^1(C))$, 
we take the local expansions 
$$t_i(x) = \left(\sum_{n=1}^\infty A_{i,n}^{(1)} x^n , \sum_{n=1}^\infty A_{i,n}^{(2)} x^n, \sum_{n=1}^\infty A_{i,n}^{(3)} x^n\right)$$
and pull back the invariant differentials from $J(C)$ given as in \eqref{InvDiffsJ}.  Modulo terms of degree $7$, the results look like
\begin{align*}
t_i^*(\Omega_j) &= A_{i,1}^{(j)} + 2 A_{i,2}^{(j)} x + \cdots + 6 A_{i,6}^{(j)} x^5\\
&\qquad +  \left(\gamma_{j,1} (A_{i,1}^{(1)})^7 + \gamma_{j,2} (A_{i,1}^{(2)})^7 + \gamma_{j,3} (A_{i, 1}^{(3)})^7\right)x^6 \bmod \deg 7.
\end{align*}
A direct check implies these expansions do exactly agree with the local expansions of the differentials $\omega_i$ (as they must, by the proof of 
Theorem~\ref{FormalAbCharp}).   The same thing is true in the following examples but we will not say this again
for brevity.

However, in this case, none of these is a logarithmic form and the formal structure of the Jacobian becomes clearer after
a change of basis to a basis of $H^0(C, \Omega^1(C))$ consisting of logarithmic forms.  
Recall that $\omega = \frac{df}{f}$ is logarithmic if and only if $C(\omega) = \omega$, where 
$C$ is the Cartier operator.  We find that 
$$\omega = \frac{(a + bx + cy) dx }{\partial f/\partial y}$$
satisfies $C(\omega) = \omega$ if and only if 
\begin{align*}
a^7 &= 4 a + 5 b,\\
b^7 &= 2 b + 3c,\\
c^7 &= 2 a + b + c.
\end{align*}
By a lexicographic Gr\"obner basis computation, we see that this system of equations has $343$ solutions
$(a,b,c)$ defined over extension fields of $\FF_7$.  A direct computation shows that one possible choice of three linearly independent
logarithmic forms is
\begin{align*}
\omega_1 &= \frac{4 + 6x + 5y}{\partial f/\partial y}\ dx,\\
\omega_2 &= \frac{(5 \varepsilon + 3) + (4 \varepsilon + 6) x + \varepsilon y}{\partial f/\partial y}\ dx,\\
\omega_3 &= \frac{5 \eta + \eta x + \eta y}{\partial f/\partial y}\ dx,
\end{align*}
where $\varepsilon^7 + 6\varepsilon + 2 = 0$ and $\eta^6 + 2 = 0$.
With respect to this basis for $H^0(C,\Omega^1(C))$ over the algebraic closure of $\FF_7$, 
the Cartier-Manin matrix becomes a $3 \times 3$ identity matrix and (up to isomorphism) the formal group of 
the Jacobian splits into a product of three factors, each isomorphic to a 1-dimensional formal group 
law of height $1$ (that is, isomorphic to the formal multiplicative group).  To write each factor in a fashion consistent
with Theorem~\ref{FGLawForm} above, though, we need to use a ``$7$-typical'' form of the formal multiplicative 
group.  This is obtained by reducing a characteristic-zero group law in Honda form modulo $7$.  The Honda
group law in characteristic zero is described by the logarithm
$$L(X) = X + \frac{X^7}{7} + \frac{X^{49}}{49} + \cdots. $$
The invariant differential is given by the derivative of this:
$$\Omega = 1 + X^6 + X^{48} + \cdots .$$
The group law
$$L^{-1}(L(X) + L(Y)) = X + Y + C_7(X,Y) + \cdots $$
has nonzero terms in degrees $1 + 6\ell$, $\ell \ge 0$ and all coefficients are in $\ZZ$ so the mod $7$ reduction
is the formal group law we want.  It is easy to see that the terms of degree $7$ in the formal group law are necessary
to obtain formal Abel relations modulo degree 8.

Now we consider a second curve that is, in a sense, at the opposite end of the spectrum of possibilities from the
ordinary curve we considered above.   By a direct calculation using \cite{StohrVoloch} again, the Cartier-Manin matrix for the curve 
\begin{align*}
0 = f(x,y) = &y^4 + (6x+3) y^3 + (5 x^2 + 4 x) y^2 + (2 x^3 + 5 x+1) y\\
                  &+ 6x^4 + 6x^3 + 4x^2 + x + 3 
 \end{align*}
with respect to the ``monomial'' basis \eqref{AbDiffBasis} is \emph{identically zero}.   All three of the 
differentials $\omega_1,\omega_2,\omega_3$ are exact in this case.  
\begin{figure}[h]
\begin{center}
\includegraphics[height=2in,width=4in]{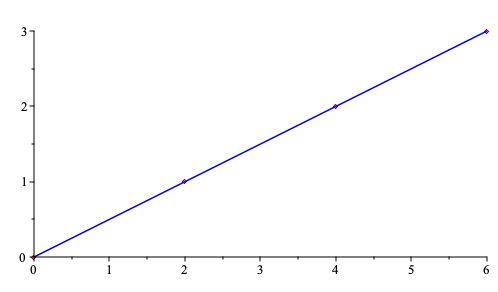}
\caption{The Newton polygon for this superspecial curve.}
\end{center}
\end{figure}
For this curve, the point counts are $N_1 = 8$, $N_2 = 92$, and $N_3 = 344$.  
Using the formulas above, the zeta function is 
$$\frac{343t^6 + 147 t^4 + 21 t^2 + 1}{(1 - t)(1 - 7t)}.$$
The Newton polygon comes from the points
$$(0,0), (2,1), (4,2), (6,3)$$
since the zero coefficients of the odd powers of $t$ can be visualized as corresponding to 
points $(1,\infty), (3,\infty), (5,\infty)$ which lie above the lines connecting the points
above from the terms with even powers of $t$.   Note that there are three lines of slope 
$m = \frac{1}{2}$, consistent with the three exact forms found above.  The corresponding Ekedahl-Oort
type is $(0,0,0)$ with $f = 0$ and $a = 3$.  

This means, of course, that 
there are functions $g_1$, $g_2$, $g_3$ such that $\omega_i = dg_i$ for $i = 1,2,3$.  
Hence there are no nonzero $x^{7\ell - 1}$ terms in the local expansions of the $\omega_i$ for any 
$\ell \ge 1$ and formal integrals of those expansions can be computed in characteristic $7$.
But in order for those formal integrals of the local expansions of the $\omega_i$ to  
give the local expansions of the $t_i^{(j)}$ for $i = 1,2,3,4$ and $j = 1,2,3$, we must
be careful to include ``constant of integration'' terms $c_{7\ell} x^{7\ell}$ for all $\ell$ in order to 
find the actual $t_i^{(j)}$.  When these are included, it is not difficult to show that the 
formal Abel relation can be continued to terms of degree $48$ with no other changes (in particular
no additional terms in the formal group law apart from the $X + Y$ in degree 1).  
However, at degree $49$,  we find that additional terms are necessary in the formal group 
law.  \emph{Each component of the formal group law} looks in fact like 
$$X + Y + \alpha C_{49}(X,Y) \bmod \langle X,Y\rangle^{50},$$
where $\alpha \ne 0$ and 
$$C_{49}(X,Y) = \frac{X^{49} + Y^{49} - (X + Y)^{49}}{7}  \mod 7.$$
Then other nonzero terms will appear in higher degrees as in the previous example in a 
Honda form $7$-typical group law.  This means each component is isomorphic to a one-dimensional formal group law of 
\emph{height 2}.   Hence, according to the usual terminology (e.g. from \cite{AP})  this is an 
example of a \emph{superspecial} (non-hyperelliptic) curve of genus $g = 3$.  The Jacobian
of the curve (i.e. the actual abelian variety, not the formal group law)
is isomorphic to a product of three supersingular elliptic curves and the formal group 
law and the Abel relation reflects this structure.  

For the purposes of comparison, we will also provide examples of smooth plane quartics over $\FF_7$ whose
Cartier-Manin matrices have the other possible ranks: $1$ and $2$.   In each case we will 
use the basis for $H^0(C,\Omega^1(C))$ given in \eqref{AbDiffBasis} to compute the Cartier-Manin matrix to start. These curves have
formal Jacobians with different structures and Abelian differentials with different properties.  
First, using \cite{StohrVoloch},  the curve
$$0 = f(x,y) = y^4 + (2x+3)y^3 + (6x^2 + 5x)y^2 + (6x^3 + 2x+1)y + 5x^4 + 6x^3 + 3$$
has Cartier-Manin matrix
$$\mathcal{C} = \begin{pmatrix} 5&  2 & 6\\ 3& 1& 1\\ 0& 2& 5\end{pmatrix}$$
of rank $2$ over $\FF_7$.   The exact form 
$$\eta = 4 \omega_1 + \omega_2 + \omega_3$$
comes from a basis of the kernel over $\FF_7$.   It is not difficult to 
find two logarithmic forms that, together with this exact form, form a basis for $H^0(C,\Omega^1(C))$.  
The point counts in this case are $N_1 = 11$, $N_2 = 67$ and $N_3 = 380$.
The zeta function works out to 
$$\frac{343 t^6 + 147 t^5 + 91 t^4 + 42 t^3 + 13 t^2 + 3 t + 1}{(1 - t)(1 - 7t)}$$
so the Newton polygon arises from the points
$$(0,0), (1,0), (2,0), (3,1), (4,1), (5,2), (6,3).$$
\begin{figure}[h]
\begin{center}
\includegraphics[height=2in,width=4in]{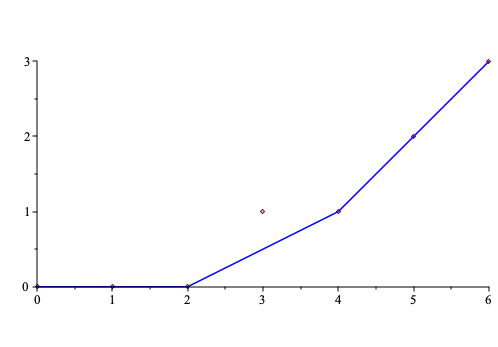}
\caption{The Newton polygon of this curve with $f = 2$.}
\end{center}
\end{figure}
There are two pairs of lines of slopes $0,1$ and one of slope $\frac{1}{2}$, consistent
with the comments above about logarithmic and exact forms on this curve.  
The Ekedahl-Oort type of the Jacobian is $(1,2,2)$ in this case.

Finally, using \cite{StohrVoloch} again, we find that the curve 
\begin{align*}
0 = f(x,y) = &y^4 + (2x + 3)y^3 + 2x^2y^2 + (x^3 + 4x^2 + 3x+1) y\\
 &+ 4 x^4 + 5 x^3 + x^2 + 2 x + 3
\end{align*}
has Cartier-Manin matrix with respect to the basis \eqref{AbDiffBasis} 
$$\mathcal{C}=\begin{pmatrix} 5& 6& 6\\  3 &  5&  5\\ 6 & 3 & 3 \end{pmatrix},$$
which has rank $=1$ over $\FF_7$.  The form of this matrix implies that both of the differentials
$$\omega_1 - 2\omega_2 \text{ and } \omega_1 - 2\omega_3$$
are exact forms since
$$\mathcal{C}(\omega_1 - 2\omega_2) = 0 = \mathcal{C}(\omega_1 - 2\omega_3).$$
It is also possible to determine a logarithmic form which completes a basis of 
$H^0(C,\Omega^1(C))$.
The point counts are $N_1 = 10$, $N_2 = 60$, $N_3 = 352$ and the zeta function 
works out to 
$$\frac{343 t^6 + 98 t^5 + 49 t^4 + 14 t^3 + 7 t^2 + 2 t + 1}{(1 - t)(1 - 7t)}.$$
The Newton polygon comes from the points
$$(0,0), (1,0), (2,1), (3,1), (4,2), (5,2), (6,3).$$
Since the points $(2,1)$ and $(4,2)$ lie above the lines between the points on 
either side, there are two line segments of slope $m = \frac{1}{2}$ and one
pair of slopes $0,1$.  Hence $f = 1$ and $a = 2$.  
\begin{figure}[h]
\begin{center}
\includegraphics[height=2in,width=4in]{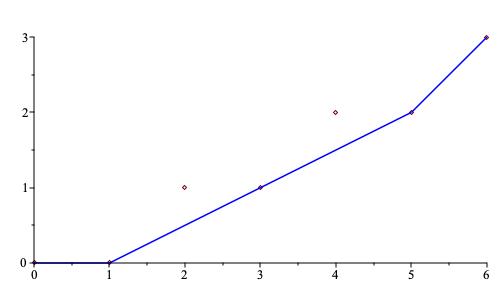}
\caption{The Newton polygon of this curve with $f = 1$.}
\end{center}
\end{figure}This is consistent with the statements above about
exact and logarithmic forms on this curve.  The corresponding Ekedahl-Oort
type is $(1,1,1)$.

%What these computations show is that the Cartier-Manin matrix determines the $x^6$
%coefficients in the expansions of the Abelian differentials, which
%are the pull-backs of the invariant differentials on the formal group of the Jacobian.  
%For example, the first curve has an ordinary Jacobian, so none of the Abelian differentials is exact.  
%(N.B.  if some differential were exact, that would imply the coefficients of  $x^6$ in its local expansions would all be zero.)  

Finally, although we have been focusing on the case of smooth quartic curves, for completeness
we discuss an Abel relation on a singular quartic curve $C$ since these are also included in Theorem~\ref{AbClassical}
when we work over $\CC$.  The curve defined by 
\begin{align*}
0 = f(x,y) &= y^4 + (x + 6) y^3 + (3x^2 + 6x + 4) y^2 + (4 x^3 + 5 x^2 + 5 x + 5) y \\
&+ 2 x^4 + x^3 + 5 x^2 + x
\end{align*}
over the algebraic closure of $\FF_7$ is an irreducible quartic with a cuspidal singularity isomorphic
to the point at the origin on $y^4 - x^3 = 0$ located at the point $P: (x,y) = (3,6)$.  The curve is in 
good position relative to the $y$-axis, but two of the intersections are only defined over the quadratic
extension of $\FF_7$ because $f(0,y)$ factors in $\FF_7[y]$ as
$$f(0,y) = y^4 + 6y^3 + 4 y^2 + 5y = y(y+4)(y^2 + 2y + 3).$$
%It is easy to check, via the formula from \cite{StohrVoloch}, that the matrix of the Cartier operator on the monomial basis
As is the case for smooth quartics, the 
$$\omega_1 = \frac{dx}{\partial f/\partial y}, \quad \omega_2 = \frac{x dx}{\partial f/\partial y}, \quad \omega_3 = \frac{y dx}{\partial f/\partial y}$$
are a basis for the global sections of the sheaf of dualizing differentials on $C$.
Moreover, we can check that the Cartier operator applied to each of the $\omega_i$ is zero.  This shows that 
that all of the basis dualizing differentials are exact.

It is also the case that  the formal group law of the (generalized) Jacobian
is isomorphic to a product of three copies of the one-dimensional formal additive group law.   Here is an argument for 
this claim based on the results from Exercise II.6.9 of \cite{Har}.  The curve $C$ has normalization $\tilde{C} \simeq \PP^1$ and
the cusp at $P$ corresponds to a unique point on $\PP^1$.
The following exact sequence relates the Picard groups of $C$ and $\tilde{C}$:
$$0 \longrightarrow \widetilde{\mathcal{O}_P}^*/\mathcal{O}_P^* \longrightarrow {\rm Pic}(C) \longrightarrow {\rm Pic}(\tilde{C}) \longrightarrow 0.$$
Passing to the formal completions we have $\widetilde{\mathcal{O}_P} \simeq k[[t]]$ and $\mathcal{O}_P \simeq k[[t^3,t^4]]$.  Hence the quotient
is isomorphic to the $3$-dimensional linear group with group operation given by 
\begin{align*}
(1 + & a t + b t^2 + c t^5) \cdot (1 + a' t + b' t^2 + c' t^5) \\
&\equiv (1 + (a + a') t + (b + b' + a a') t^2 + (c + c') t^5) \bmod k[[t^3,t^4]].
\end{align*}
Using the invariant differentials, is not difficult to show (since ${\rm char}(k) \ne 2$) that the corresponding formal group law:
$$G(X,Y) = (X_1 + Y_1, X_2 + Y_2 + X_1 Y_1, X_3 + Y_3)$$
has a logarithm defined by 
$$\mathcal{L}(X) = (X_1, X_2 - \frac{X_1^2}{2}, X_3)$$
satisfying
$$\mathcal{L}^{-1}(\mathcal{L}(X) + \mathcal{L}(Y)) = G(X,Y).$$
This shows that the formal group of the generalized Jacobian is isomorphic to $\mathbb{G}_a^3$.  

Hence we find
the mappings $t_i^{(j)}(x)$ by formally integrating the local expansions of the $\omega_j$, including suitable ``constants of integration''
containing powers $x^{7\ell}$ for $\ell \ge 1$.  The coefficients of those terms can be chosen to make the Abelian sums identically zero, and this
yields a formal Abel relation.  

To summarize, from an intuitive point of view, we have seen in these examples how formal Abel relations on algebraic
curves can be found in a number of different cases.  The formal group structure of the (generalized) Jacobian plays
a key role in finite characteristic and there are many more possibilities than in characteristic zero.  For singular
curves, there are many additional possible forms involving components isomorphic to formal additive groups that
we have not included.  However, the series expansions of the components of the mappings  
$t_i$ still do come from a sort of ``integral'' defined in finite characteristic.   It's the Cartier operator that makes up
for the fact that the ordinary derivative of  $x^p$ is $0 \bmod p$, 
so that no derivative of a power series over $\FF_p$ can contain nonzero $x^{p-1}$ terms. 

\section{Observations about the non-uniqueness of the mappings $t_i : C_i \to G$ in formal Abel relations}

We have seen how the choice of a basis in $H^0(C,\Omega^1(C))$ produces different mappings $t_i : C_i \to G$, 
where $G$ is the formal completion of the (generalized) Jacobian of $C$ at the origin.  In this section we
want to point out that there is an additional level of non-uniqueness here that appears only in positive characteristic.
This comes ultimately from a certain sort of $p$-adic periodicity of the equations defining the maximal Abel relations.

To begin, suppose the curve $C$ is defined over a finite field $\FF_q$.
In particular, say the $C_i$ are defined by $y = \sum_{n=0}^\infty a_{i,n} x^n$  with  $a_{i,n}$ in the finite field $\FF_q$ for all $n$.
Then the equations from \S~\ref{FirstSteps} above whose solutions are the $A_{i,k}^{(j)}$ in maximal formal Abel relations
are linear equations with coefficients in $\FF_q$.  Hence as we have done in the examples, we may assume that 
all the $A_{i,k}^{(j)} \in \FF_q$ as well.  In other words, in this case we will have
$$(a_{i,n})^q = a_{i,n} \quad \text{ and } \quad (A_{i,k}^{(j)})^q = A_{i,k}^{(j)}$$
for all $i, n, k, j$.   Moreover, let us assume that the $g$-dimensional formal group law $G$ is also defined over $\FF_q$.  This
implies, of course, that raising each component of the $g$-tuple of power series to the $q$th power, we obtain an equality
\begin{equation}
\label{GoverFq}
(G(X,Y))^q = G(X^q,Y^q).
\end{equation}

For simplicity, say we have $d = 4$ curves and a maximal Abel relation in a $(4 - 1)(4 - 2)/2 = 3$-dimensional formal group
law $G$.  Explicitly, this can be written in the form
\begin{equation}
\label{FAR}
0 = \prod_G t_i(x_i) = G(G(G(t_1(x_1),t_2(x_2)), t_3(x_3)), t_4(x_4)),
\end{equation}
where the $x_i = x_i(u,v)$ are the power series determined before from the curves $C_i$. 
Raising both sides of the equation in \eqref{FAR} to the $q$th power and using \eqref{GoverFq}, we obtain
\begin{equation}
\label{FARqthpower}
0 = G(G(G((t_1(x_1))^q,(t_2(x_2)))^q, (t_3(x_3)))^q, (t_4(x_4))^q).
\end{equation}
But writing the components of $t_i(x)$ (not $t_i(x_i)$ yet!) as 
$$t_i^{(j)}(x) = \sum_{k=1}^\infty A_{i,k}^{(j)} x^k,$$
when the $A_{i,k}^{(j)} \in \FF_q$ we have 
\begin{align}
\label{tiqthpower}
(t_i^{(j)}(x))^q &= \left(\sum_{k=1}^\infty A_{i,k}^{(j)} x^k\right)^q \nonumber \\
&=\sum_{k=1}^\infty A_{i,k}^{(j)} (x^q)^k \\
&= t_i^{(j)}(x^q).\nonumber
\end{align}

In addition, note that since we are assuming $a_{i,n} \in \FF_q$
for all $i,n$, it follows that 
\begin{equation}
\label{xqthpower}
(x_i(u,v))^q = x_i(u^q,v^q).
\end{equation}
Hence combining \eqref{tiqthpower} and \eqref{xqthpower}, 
$$(t_i(x_i))^q = t_i ((x_i(u,v))^q) = t_i(x_i(u^q,v^q)).$$

Now, on the other hand, consider the quantities 
$$B_{i,m}^{(j)} = \begin{cases} A_{i,m/q}^{(j)} \text{ if } q | m, \text{ that is, } m = n\cdot q \text{ for some integer } n \\
                                                       0 \text{ otherwise,}
                           \end{cases}
$$
and define 
\begin{equation}
\label{sdef}
s_i(x) = \left(\sum_{m=1}^\infty B_{i,m}^{(j)} x^m : j = 1,2,3\right) =  \left(\sum_{n=1}^\infty A_{i,n}^{(j)} x^{n\cdot q} : j = 1,2,3\right).
\end{equation}                      
 (Note:  I believe this is related to the so-called Frobenius operator ${\bf f}_q$ in the Dieudonn\'e module of the formal group law.)
 We have 
 \begin{align}
 \label{sixi}
 s_i(x_i(u,v)) &= \left(\sum_{m=1}^\infty B_{i.m}^{(j)} (x_i(u,v))^m : j = 1,2,3\right) \nonumber \\
 &= \left(\sum_{n=1}^\infty A_{i.n}^{(j)} ((x_i(u,v))^q)^n : j = 1,2,3\right)\nonumber \\
 &=  \left(\sum_{n=1}^\infty A_{i.n}^{(j)} (x_i(u^q,v^q))^n : j = 1,2,3\right).
 \end{align}
 From this we conclude the following.
 
 \begin{prop}
 With the $s_i(x)$ defined in \eqref{sdef}, it is true that
 \begin{equation}
 \label{FARs}
 \prod_G s_i(x_i(u,v)) = 0
 \end{equation}
 and hence we have another formal Abel relation on the same curves $C_i$.
 \end{prop}
 
 \begin{proof}
From \eqref{sixi}, the left side in \eqref{FARs} is equal to the original formal Abel relation 
$$\prod_G t_i(x_i(u,v))$$
with $u$ replaced by $u^q$ and $v$ replaced by $v^q$
everywhere.   Hence the result must also equal $0$ identically in $u,v$.
But this says that  \eqref{FARs} \emph{is another formal Abel relation} on the same curves $C_i$ but with the mappings $s_i$
instead of the mappings $t_i$.         
 \end{proof}        
 
If we want to require that $t_i'(0)$ must be nonzero in formal Abel relations, then that condition can be recovered by 
replacing $s_i$ by $r_i = G(t_i, s_i)$.   A direct consequence is another level of non-uniqueness.  
 
 \begin{prop}  
 \label{Comp}
 If the $t_i$ and the $s_i$ as in \eqref{sdef} appear in
 formal Abel relations on the $C_i$, then the mappings
 $$r_i = G(t_i,s_i)$$
 also give a formal Abel relation on the $C_i$.
 \end{prop}           
 
 \begin{proof}  
 This is a consequence of the associativity and commutativity of $G$.  We can rearrange
 the terms in 
 $$\prod_G r_i(x_i) = \prod_G G(t_i,s_i) (x_i) = \prod_G G(t_i(x_i), s_i(x_i))$$
 to put all of the $t_i(x_i)$ and all of the $s_i(x_i)$ terms together.  In other words
 $$  \prod_G G(t_i, s_i) (x_i) = G\left( \prod_G t_i(x_i), \prod_G s_i(x_i)\right) = G(0,0) = 0.$$
 This shows that the $r_i = G(t_i, s_i)$ also give formal Abel relations on the $C_i$.
 \end{proof}
 
 There are corresponding statements, of course, for mappings $r_i$ formed using $G$ to  combine
 $t_i$ and any collection of $s_i$ formed as above, but with $q$ replaced by $q^\ell$ for any strictly
increasing sequence of $\ell \ge 1$.
 
  In fact, very similar arguments also apply if $C$ is defined over an algebraically closed field $K$ of characteristic $p$, not only over a finite field
 $\FF_q$ with $q = p^r$ as we were assuming above.  Namely, if
 $$t_i^{(j)}(x) = \sum_{k=1}^\infty A_{i,k}^{(j)} x^k,$$
 with $A_{i,k}^{(j)} \in K$ for all $i,k$, then 
 $$(t_i^{(j)}(x_i))^p = \sum_{k=1}^\infty \left(A_{i,k}^{(j)}\right)^p (x_i^p)^k.$$
 Hence if we define
 $$s_i(x) = \left(\sum_{m=1}^\infty B_{i,m}^{(j)} x^m : j = 1,2,3\right),$$
where 
 $$B_{i,m}^{(j)} = \begin{cases} \left(A_{i,m/p}^{(j)}\right)^p \text{ if } p | m, \text{ that is, } m = p\cdot k \text{ for some integer } k \\
                                                       0 \text{ otherwise,}
                           \end{cases}
$$
then 
$$s_i(x_i) = (t_i(x_i))^p.$$
Now if we agree to take the formal group law $G$ as in Theorem~\ref{DieudonneManin}, the fact that
$G$ is defined over $\FF_p$ implies that 
$$\prod_G s_i(x_i) = \prod_G (t_i(x_i))^p = \left(\prod_G t_i(x_i)\right)^p = 0,$$
and we have another formal Abel relation on the curves $C_i$.  Similarly to the result of Proposition~\ref{Comp} above, 
we also have
$$\prod_G r_i(x_i) = 0, \quad \text{ where } r_i = G(t_i,s_i).$$

\section{Counterexamples to the most general converse statement}
\label{CounterEx}

In this section we will present an example showing that there are isolated counterexamples to the 
converse of Abel's theorem from Theorem~\ref{AbConv} over fields of characteristic $p$.  These
come specifically in situations where the
%the Additional Assumption~\ref{AA} does not hold and three or more
%of 
tangent lines of the curves $y = f_i(x)$ at the points $(0,a_{i,0})$ are concurrent.  In particular, 
we will consider the case where 
$$y = f_i(x) = a_{i,0} \bmod \langle x^5\rangle$$
so the four tangent lines are parallel, meeting at a point at infinity.   (Curves projectively 
equivalent to this one can be treated in a parallel way, but the formulas are significantly messier.)
The reducible quartics
\begin{equation}
\label{ReducibleCurve}
0 = f(x,y) = (y - a_{1,0})(y - a_{2,0})(y - a_{3,0})(y - a_{4,0})
\end{equation}
have formal Abel relations in the formal group law equal to the product of three
copies of the formal additive group:
$$G(X,Y) = X + Y$$
where $X = (X_1,X_2,X_3)^t$ and similarly $Y = (Y_1,Y_2,Y_3)^t$.
In addition, in constructing the counterexamples, we will suppose that the
$a_{i,0}$ for all $i = 1,2,3,4$ are elements of a finite field $\FF_q$ with $q = p^r$ 
and the prime $p \ge 7$, while $r \ge 1$.

Note that the new deformed formal curves
\begin{equation}
\label{DefBranches}
y = \widetilde{f_i}(x) = a_{i,0} + x^q
\end{equation}
(that is, $a_{i,q} = 1$ for all $i = 1,2,3,4$)
cannot come from an algebraic quartic curve (although it is true that they would lie on an algebraic curve of degree $4q$).
However, the construction we will present can also be ``jazzed up'' in such a way that the nonzero deformation
terms include infinitely many nonzero coefficients.

\begin{lem}  Computing the corresponding $x_i(u,v)$ for the deformed branches from \eqref{DefBranches}, 
we find
$$x_i(u,v) = \sum_{k=0}^\infty a_{i,0}^{q^k} u^{q^k+q^{k-1}+\cdots+q+1} + u^{q^{k-1} + \cdots + q+1} v^{q^k}.$$
\end{lem}

\begin{proof}  This follows from the general patterns in the coefficients of $x_i(u,v)$ from \S 2 above, or by an induction 
using the equation $x = u y + v$ and repeatedly substituting $y = a_{i,0} + x^q$ to eliminate $y$.  
 To illustrate the 
meaning of the formula, we show the first few terms in the expansion for $q = p = 7$:
$$x_i(u,v) = a_{i,0} u + v + a_{i,0}^7 u^8 + u v^7 + a_{i,0}^{49} u^{57} + u^8 v^{49} + a_{i,0}^{343} u^{400} + u^{57} v^{343} + \cdots .$$
From the point of view of the formulas from \S 2, there are infinitely many terms since the powers $a_{i,q}^{q^k} = 1$ appear 
not multiplied by products of the other
coefficients $a_{i,n}$ with $n \ge 1$ in infinitely many of the coefficients of $x_i(u,v)$.
\end{proof}

\begin{thm} 
The curves from \eqref{DefBranches} have a formal Abel relation using the same formal additive group
as the reducible curve and also the same $t_i^{(j)}(x)$ as the reducible curve.  
\end{thm}

\begin{proof}
By our general formulas from before, for the reducible quartic, 
\begin{align*}
t_i^{(1)}(x) &= \int \frac{dx}{\partial f/\partial y} = \frac{1}{\delta_i} x,\\
t_i^{(2)}(x) &= \int \frac{x dx}{\partial f/\partial y} = \frac{1}{2 \delta_i} x^2,\\
t_i^{(3)}(x) &= \int \frac{a_{i,0} dx}{\partial f/\partial y}  = \frac{a_{i,0}}{\delta_i} x,
\end{align*}
where (using the $f$ from \eqref{ReducibleCurve} above), 
$$\delta_i = \left.\frac{\partial f}{\partial y}\right|_{y = a_{i,0}} = \prod_{j \ne i}(a_{i,0} - a_{j,0}).$$
It follows from the assumption $a_{i,0} \in \FF_q$ that $a_{i,0}^{q^k} = a_{i,0}$ 
for all $i =1,2,3,4$ and all $k \ge 0$.   

In the sums 
\begin{align*}  
\sum_{i=1}^4 A_{i,1}^{(1)} x_i(u,v) &= \sum_{i=1}^4  \frac{{\displaystyle \sum_{k=0}^\infty a_{i,0}^{q^k} u^{q^k+q^{k-1}+\cdots+q+1} + u^{q^{k-1} + \cdots + q+1} v^{q^k}}}{\delta_i}\\
\sum_{i=1}^4 A_{i,1}^{(2)} (x_i(u,v))^2 &= \sum_{i=1}^4 \frac{{\displaystyle \sum_{k=0}^\infty (a_{i,0}^{q^k} u^{q^k+q^{k-1}+\cdots+q+1} + u^{q^{k-1} + \cdots + q+1} v^{q^k})^2}}{2 \delta_i}\\
\sum_{i=1}^4 A_{i,1}^{(3)} x_i(u,v) &= \sum_{i=1}^4 \frac{a_{i,0} {\displaystyle \sum_{k=0}^\infty a_{i,0}^{q^k} u^{q^k+q^{k-1}+\cdots+q+1} + u^{q^{k-1} + \cdots + q+1} v^{q^k}}}{\delta_i},
\end{align*}
we interchange the order of summation, replace all powers $a_{i,0}^{q^k}$ by $a_{i,0}$ and then separate into sums of terms with the 
same powers of $u$ and $v$.  Each coefficient of a $u^m v^n$ 
will then have the form a constant times
\begin{equation}
\label{sums}
\sum_{i=1}^4 \frac{a_{i,0}^\ell}{\delta_i}
\end{equation}
with $\ell = 0, 1$, or $2$.  However, if $f(y) = \prod_{i=1}^k (y - b_i)$ is any polynomial with distinct roots $b_i$, then considering
the values $\delta_i = f'(b_i)$ of the derivative of $f$, 
$$\sum_{i=1}^k \frac{b_i^\ell}{\delta_i} = \begin{cases}  0  &  \text{ if } 0 \le \ell \le k - 2\\
                                                                                     1 &  \text{ if }  \ell = k - 1, 
                                                                                     \end{cases}$$
(and the sum is equal to the complete homogeneous symmetric function of the $b_i$ of degree $\ell - 3$ for $\ell \ge k$).\footnote{It is interesting to note that the same relations (which are related to the so-called Euler-Jacobi formula and which
can be derived from the Lagrange interpolation formula) were also used in a crucial
way by Abel in his famous ``Paris memoir.''  See the footnote 50 on page 31 of \cite{LitAbel}.}
Hence all of the sums in \eqref{sums} with $\ell = 0,1,2$ are zero and we have a formal Abel relation.
\end{proof}

\bibliographystyle{plain}

\end{document}